\documentclass[reqno]{amsproc}
\usepackage{amssymb}
\usepackage{amsmath}
\usepackage{amsfonts}

\theoremstyle{plain}

\newtheorem{corollary}{Corollary}

\newtheorem{definition}{Definition}
\newtheorem{example}{Example}

\newtheorem{lemma}{Lemma}

\newtheorem{proposition}{Proposition}
\newtheorem{remark}{Remark}

\newtheorem{theorem}{Theorem}
\numberwithin{equation}{section}

\renewcommand{\leq}{\leqslant}
\renewcommand{\ge}{\geqslant}
\renewcommand{\geq}{\geqslant}

\begin{document}

\title[ On functions with Fourier transforms in Generalized  Grand Lebesgue space ]{On functions with Fourier transforms in Generalized  Grand Lebesgue space}
\author{A.Turan G\"{u}rkanl\i}
\address{A.Turan G\"{u}rkanl{\i} \\
Istanbul Arel University Faculty of Science and Letters\\
Department of Mathematics and Computer Sciences\\
\.Istanbul\\
Turkey}
\email{turangurkanli@arel.edu.tr}
\author{ B\"{u}lent Ayanlar}
\address{B\"{u}lent Ayanlar \\
Istanbul Arel University Faculty of Science and Letters\\
Department of Mathematics and Computer Sciences\\
\.Istanbul\\
Turkey}
\email{bulentayanlar@arel.edu.tr}
\author{ Esra Uluocak}
\address{Esra Uluocak \\
Istanbul Arel University Faculty of Science and Letters\\
Department of Mathematics and Computer Sciences\\
\.Istanbul\\
Turkey}
\email{esraaltinbilezik@arel.edu.tr}
\subjclass[2000]{Primary 46 E 30; Secondary 46 E 35.}

\begin{abstract}

 Let $1<p,q<\infty ,\ \theta_1 \geq 0,\ \theta_2 \geq 0$
 and let $a(x), b(x)$ be a weight functions. 
 In the present paper we intend to study
 the function  space $A_{q),\theta _{2}}^{p),\theta _{1}}\left( \mathbb R^n\right)$ consisting of all functions $f\in L_a^{p),\theta_1 }\left( \mathbb R^n\right) $ whose generalized Fourier transforms
$\widehat{f}$ belong to grand $L_b^{q),\theta_2 }\left( \mathbb R^n\right), $ where $ L_a^{p),\theta_1 }\left( \mathbb R^n\right)$ and $L_b^{q),\theta_2 }\left( \mathbb R^n\right) $ are generalized grand Lebesgue spaces. In the second section some definitions and notations  used in this work are given. In the third and fourth sections we discuss some basic properties and inclusion properties of $A_{q),\theta _{2}}^{p),\theta _{1}}\left( \mathbb R^n\right)$. In the fifth section we characterize the multipliers from  ${L^{1 }(\mathbb R^{n}, a^{\frac{\varepsilon}{p}})}$  to  $( L_a^{p),\theta}\left(\mathbb R^{n}\right))^{\ast}$ and from ${L^{1 }(\mathbb R^{n}, a^{\frac{\varepsilon}{p}})}$   into  $(A_{q),\theta_2}^{p),\theta_1}\left(\mathbb R^{n}\right))^{\ast}$  for ${0<\varepsilon \leq p-1}.$ The importance of this section is that, it gives us some insight  into the structure of the dual space $( L_a^{p),\theta}\left(\mathbb R^{n}\right))^{\ast}$ of the generalized grand Lebesgue space, the properties of which are  not yet known. Later we discuss duality and reflexivitiy properties of the space $A_{q),\theta _{2}}^{p),\theta _{1}}\left( \mathbb R^n\right)$.
\end{abstract}

\maketitle
\section{\protect\bigskip Introduction}
For $1\leq p\leq\infty,$ we denote by $ A^{p}\left(\mathbb R^{n}\right)$
 the space of functions $f\in L^1(\mathbb R^{n})$ for which $\hat f$  Fourier transform of $ {f}$ exists and  belongs to   $L^{p)}\left( \mathbb R^{n}\right).$ 
Research on $ A^{p}\left(\mathbb R^{n}\right) $ was initiated by Warner \cite{wr}, a number of researchers such as e.g., Larsen, Liu and Wang \cite{llw}, Martin and Yap \cite{my} worked on this spaces. Some generalizations to the weighted case was given by Gurkanli \cite{g1}, Feichtinger and Gurkanli \cite{fg}, Fischer, Gurkanli and Liu  \cite {rgl}. Also the multipliers and compact embeddings of the weighted  $ A^{p}\left(\mathbb R^{n}\right) $ and $ A^{p}\left(\mathbb R^{n}\right) $  spaces were worked in Gurkanli \cite{g2}, Tewari,Gupta\cite{tegu} and Gurkanli \cite {g3} respectively. In all these studies, Lebesgue space, weighted Lebesgue space and generalized Fourier transform have been used. In the present study, it was redefined using the generalized grand Lebesgue space (see Samko, Umarkhadzhiev \cite{su1},\cite{su3} and  Umarkhadzhiev  \cite{u}). It was denoted by $A_{b,q)}^{a,p)}\left(\mathbb R^{n}\right)$. 


The initial definition of \textit{grand Lebesgue spaces assumes $\mid \Omega\mid<\infty$}. This space was introdeced by Iwaniec-Sbordone in \cite{is}, and later an extension was given by Greco-Iwaniec-Sbordone in \cite{gis}. Grand Lebesgue spaces on bounded sets  $ \Omega\subset \mathbb R^{n}$ are widely studied, see e.g.\cite{cfg}, \cite{fk},\cite{fi},\cite{g4} and \cite{g5}. In \cite{u} and \cite{su3} Samko and Umarkhadzhiev consider a more common approach to definition of grand Lebesgue spaces by an arbitrary open set $\Omega\subseteq \mathbb R^{n}.$ They define the generalized grand Lebesgue space  ${L_a^{p),\theta}(\Omega)}$ on a set $\Omega$ of possibly infinite measure.

\section{\protect\bigskip Preliminaries}
In this paper we will work on $\mathbb R^{n}$ with Lebesgue measure $dx.$ We denote by $\left \vert A\right\vert$ the Lebesgue measure of a measurable set $A\subseteq\mathbb R^{n} .$ We denote by $C_{c}(R^{n})$ the space of complex-valued, continuous functions with compact support and we denote by $C_{0}^{\infty}(\mathbb R^{n})$ the space of infinitely differentiable complex-valued continuous functions with compact support. The translation and modulation operators are given by%
\begin{align*}
T_{x}f\left( t\right) =f\left( t-x\right) ,\text{ }M_{\xi }f\left( t\right)
=e^{ i<\xi, t>}f\left( t\right) ,\text{ }t,\text{ }x,\text{ }\xi \in\mathbb R^{n}.
\end{align*}
 A positive, measurable and   locally integrable function $\omega$ vanishing on a set of zero measure is called a \textit{ weight function}. We say that $\omega_1\prec\omega_2$  if there exists a constant $C>0$ such that $\omega_1(x)\leq C\omega_2(x)$. Two weight functions are called equivalent  and written $\omega_1\approx\omega_2,$ if $\omega_1\prec\omega_2$ and $\omega_2\prec\omega_1.$ A weight function $\omega$ is called \textit{ submultiplicative}, if
\begin{equation}
 \omega(x+y)\leq\omega(x)\omega(y),\ \ \forall x,y\in \mathbb R^{n}.
\end{equation}
  A weight function $\omega$ is called \textit{Beurling's weight function} on $\mathbb R^n$ if submultiplicative and $\omega(x)\ge1, $ \cite{re}.  For $1\leq p<\infty,$ we define the weighted space $L_{\omega}^p(\mathbb R^{n})$ with the norm
\begin{align}
\left\Vert {f}\right\Vert_{{L_{\omega}^p}}=\left( \int_{\mathbb R^{n} }\left\vert f\right\vert ^{p }\omega(x)dx\right) ^{\frac{1}{p }},  
\end{align}
 \cite{cr},\cite{g1},\cite{re}. The main tool in the present paper is the Fourier transform, (see\cite {k},\cite {ru}), and the generalized Fourier transform  discussed below, (see \cite{h},\cite{t}).
For a function $f\in L^{1}(\mathbb R^{n}) ,$ the function $\widehat{f}$  or ($\mathcal{F}(f))$
defined on $\widehat{\mathbb R^{n}}=\mathbb R^{n}$ by $\ $ 
\begin{equation}
\hat{f}\left( \gamma \right) =\int_{\mathbb R^{n}}f\left( x\right) \ e^{-i<\gamma, x>} dx,\qquad (\gamma \in \widehat{\mathbb R^{n}}=\mathbb R^{n}),
\end{equation}
is called the \textit{\ Fourier transform} of $f,$ where $\widehat{\mathbb R^{n}}$ is the Pontryagin dual of  $\mathbb R^{n}.$ It is known that $\hat{f}$ is uniformly continuous and   $\hat{f}\in C_0( \widehat{\mathbb R^{n}})=C_0(\mathbb R^{n}),$ where $C_0(\mathbb R^{n})$  is the space of continuous functions on $\mathbb  R^{n}$ which vanish at infinity. 

 For $n\in\mathbb N, $
\begin{equation}
\mathcal{S}(\mathbb R^{n}) =\{f\in \mathcal{C^{\infty}}(  \mathbb R^{n})\mid \sup_{x\in\mathbb R^{n}}(1+\mid x\mid^2)^{\frac{k}{2}}\sum_{\mid \alpha\mid\leq\ell}\mid D^{\alpha}f(x)\mid<\infty, for\ all\ k,\ell \in\mathbb N \}
\end{equation} 
 is called Schwartz space  \cite{h, t} . The elements of this space is rapidly decreasing infinitely differentiable functions. A sequence $(f_n)_{n\in N}\subset\mathcal{S}(\mathbb  R^{n}) $ is said to converge in $\mathcal{S}( \mathbb R^{n}) $  to a function $f\in\mathcal{S}( \mathbb R^{n}) $ if 
\begin{equation*}
\|f_n-f\|_{k,\ell}\rightarrow0, for\ all \ k,\ell\in\ \mathbb N
\end{equation*}
 where
\begin{equation}
\|f\|_{k,\ell}= \sup_{x\in\mathbb R^{n}}(1+\mid x\mid^2)^{\frac{k}{2}}\sum_{\mid \alpha\mid\leq\ell}\mid D^{\alpha}f(x)\mid.
\end{equation}
 It is well-known that the Fourier transforms induce an isomorphism between $\mathcal{S}\left(  R^{n}\right) $ and $\mathcal{S}\left( \widehat{\mathbb R^{n}}\right) ,$ hence their
transposes induce isomorphism of their dual space . Thus the  \textit{ generalized Fourier transform} is defined by $%
\left\langle \widehat{\sigma },f\right\rangle =\left\langle \sigma ,\widehat{%
f}\right\rangle $ \ for \ $f\in\mathcal{S}\left( \widehat{\mathbb R^{n}}\right) ,\sigma \in
\mathcal{S}^{^{\prime }}\left(\mathbb R^{n}\right) ,$ \cite{ t,h}  . It is known that $%
L^{p}\left(\mathbb  R^{n}\right) \hookrightarrow \mathcal{S}^{^{\prime }}\left(\mathbb  R^{n}\right) $ $%
 .$ We shall write $\hat{f}$ \ or $\mathcal {F}(f)$ for Fourier transform of the
function $f\in L^{p}\left(\mathbb R^{n}\right) .$

Let $\Omega$ be an open subset of $ R^{n}$and $|\Omega|<\infty.$  The \textit{grand
Lebesgue space} $L^{p)}\left( \Omega \right) $ is defined by the norm
\begin{equation}
\left\Vert {f}\right\Vert _{p)}=\sup_{0<\varepsilon \leq p-1}\left(
\varepsilon \int_{\Omega }\left\vert f\right\vert ^{p-\varepsilon }d\mu
\right) ^{\frac{1}{p-\varepsilon }},
\end{equation}
where $1<p<\infty, $ \cite{ is}. A kind of \textit{ generalization of the grand Lebesgue spaces}
are the spaces $L^{p),\theta }\left( \Omega \right) ,$ $\theta \geq 0,$
defined by the norm
\begin{align}
\notag\left\Vert f\right\Vert _{p),\theta ,\Omega }=\left\Vert f\right\Vert
_{p),\theta }&=\sup_{0<\varepsilon \leq p-1}\varepsilon ^{\frac{\theta }{
p-\varepsilon }}\left( \int_{\Omega }\left\vert f\right\vert ^{p-\varepsilon
}d\mu \right) ^{\frac{1}{p-\varepsilon }}\\&=\sup_{0<\varepsilon \leq
p-1}\varepsilon ^{\frac{\theta }{p-\varepsilon }}\left\Vert {f}\right\Vert
_{p-\varepsilon },
\end{align}
when $\theta =0$ the space $L^{p),0}\left( \Omega \right) $ reduces to
Lebesgue space $L^{p}\left( \Omega \right) $ and when $\theta =1$ the space $
L^{p),1}\left( \Omega \right) $ reduces to grand Lebesgue space $
L^{p)}\left( \Omega \right),$ \cite{ gis}. For $
0<\varepsilon \leq p-1,$
\begin{equation*}
L^{p}\left( \Omega \right) \subset L^{p),\theta }\left( \Omega \right)
\subset L^{p-\varepsilon }\left( \Omega \right)\subset L^1(\Omega) .
\end{equation*}
hold. It is known that the subspace ${C_{0}^{\infty }}\left( \Omega \right) $
is not dense in $L^{p),\theta }\left( \Omega \right). $ Its closure consists of functions $f\in L^{p)}\left( \Omega \right)
$ such that
\begin{equation*}
\lim_{\varepsilon \rightarrow 0}\varepsilon ^{\frac{\theta }{p-\varepsilon }
}\left\Vert f\right\Vert _{p-\varepsilon }=0
\end{equation*}
(see \cite{ gis}). The grand Lebesgue space $L^{p),\theta }\left( \Omega \right)
,$ is not reflexive. For some properties and applications of $L^{p)}\left(
\Omega \right)$ and $L^{p),\theta }\left( \Omega \right) ,$ we refer to
papers \cite{is,a,cfg,cr,g4,g5}. 

In all above mentioned studies only sets $\Omega$ of finite measure were allowed, based on the embedding
\begin{equation*}
L^{p}\left( \Omega \right) \hookrightarrow
L^{p-\varepsilon }\left( \Omega \right) .
\end{equation*}
Let $1< p< \infty$ and $\Omega\subseteq R^n $ be an open subset. We define the \textit{ generalized grand Lebesgue space} $L_a^{p),\theta}(\Omega)$ on a set $\Omega$ of possibly infinite measure as follows (see,\cite{su3},\cite{u}): 
\begin{align}
L_a^{p),\theta}(\Omega)&=\{ f:\left\Vert {f}\right\Vert _{L^{p),\theta}_a(\Omega)} = \sup_{0<\varepsilon \leq p-1}\varepsilon^{\theta} \left( \int_{\Omega }\left\vert f(x)\right\vert ^{p-\varepsilon} a(x)^{\frac{\varepsilon}p} dx
\right) ^{\frac{1}{p-\varepsilon}
}\\&=\sup_{0<\varepsilon \leq p-1}\varepsilon^{\theta} \left\Vert {f}\right\Vert_{L^{p-\varepsilon }(\Omega, a^{\frac{\varepsilon}p})}<\infty \}.\ \notag 
\end{align}
The norm of this space is equivalent to the norm
\begin{align}
\left\Vert {f}\right\Vert _{L^{p),\theta}_a(\Omega)}& = \sup_{0<\varepsilon \leq p-1}\left(
\varepsilon^{\theta} \int_{\Omega }\left\vert f(x)\right\vert ^{p-\varepsilon} a(x)^{\varepsilon} dx
\right) ^{\frac{1}{p-\varepsilon}
}\\&=\sup_{0<\varepsilon \leq p-1}\varepsilon ^{\frac{\theta }{p-\varepsilon }}\left\Vert {f}\right\Vert
_{L^{p-\varepsilon }(\Omega, a^{\varepsilon})}\ \notag
\end{align}
(see \cite{su3}). We call $a(x)$ the grandizer of the space $L_a^{p),\theta}(\Omega)$. It is known that the embedding
 \begin{align}
L^p(\Omega)\subset L_a^{p),\theta}(\Omega)
 \end{align}
holds if and only if $a\in L^1(\Omega)$ (see \cite{su3}).   It is also known that $L_a^{p),\theta}(\Omega)$ is a Banach space  and  for $ 0<\varepsilon \leq p-1, \  p>1$ 
\begin{equation}
L^{p}\left( \Omega \right) \subset L_a^{p),\theta}(\Omega)
\subset{L^{p-\varepsilon }(\Omega, a^{\frac{\varepsilon}p})} .
\end{equation}
hold.

Let $\left( A,\left\Vert \cdot\right\Vert _{A}\right) $ be a Banach algebra. A
Banach space $\left( B,\left\Vert \cdot\right\Vert _{B}\right) $ is called a
\textit{Banach module }over a Banach algebra $\left( A,\left\Vert
\cdot\right\Vert _{A}\right) $ if $B$ is a module over $A$ in the algebraic
sense for some multiplication, $\left( a,b\right) \rightarrow a.b\ ,\ $and
satisfies $\left\Vert a.b\right\Vert _{B}\leq M \left\Vert a\right\Vert
_{A}\left\Vert b\right\Vert _{B} $ for some $M>1.$

Let $B_{1},B_{2}$ be Banach $A$-modules.Then a \textit{multiplier (or
module homomorphism)} from $B_{1}$ to $B_{2}$ is a bounded linear operator $
T $ from $B_{1}$ to $B_{2}$ which commutes with the module multiplication,
i.e. $T(ab)=aT(b)$ for all $a\in A,\  b\in B_{1}$. The space of multipliers
from $B_{1}$ to $B_{2}$ is denoted by $M(B_{1},B_{2})$ (or $\mathrm{Hom}(B_{1},B_{2})$). By Corollary 2.13 in \cite{ri1},
\begin{equation}
\mathrm{Hom}_{A}\left( B_{1},B_{2}^{\ast }\right) \cong \left( B_{1}\otimes
_{A}B_{2}\right) ^{\ast },
\end{equation}
where $B_{2}^{\ast }$ is the dual of $B$ and $\otimes _{A}$ is the $A-$
module tensor product, \cite{ri1,ri2}.


\section{The space $A_{q),\theta_2}^{p),\theta_1}\left( R^{n}\right) $ and some basic properties} 


In this section we will define a functional space $A_{q),\theta_2}^{p),\theta_1}\left( R^{n},a,b\right) $ (for short $A_{q),\theta_2}^{p),\theta_1}\left( R^{n}\right)), $ if a=b we denote by $A_{q),\theta_2}^{p),\theta_1}\left( R^{n},a\right) $     and discuss some basic properties of this space.

\begin{definition}
Let $ a(x),b(x)$ be weight functions defined on $\mathbb R^{n}$ and $\widehat{\mathbb R^n},$ respectively, where  $\widehat{\mathbb R^n}$  is the Pontryagin dual of $\widehat{\mathbb R^n}$ \cite{ru} . Assume that  $1<p,q<\infty,\ \theta_1, \theta_2>0.$  We denote by $A_{q),\theta_2}^{p),\theta_1}\left(\mathbb  R^{n},a,b\right) $ (or $A_{q),\theta_2}^{p),\theta_1}\left( \mathbb R^{n}\right), $)
the set of functions $f$ for which $\hat f$  Fourier transform of {f} exists and  belongs to   $L^{q),\theta_2}_b( \widehat{\mathbb  R^n})$ (It is known that $\widehat{\mathbb R^n}=\mathbb R^{n}).$ It is easy to show that this set is a vector space over $ \mathbb R^{n}.$  Let's equip this vector space with the norm
\begin{align}
\left\Vert f\right\Vert _{A_{q),\theta_2}^{p),\theta_1}} =\left\Vert
f\right\Vert _{L_a^{p),\theta_1}}+\left\Vert \hat{f}\right\Vert _{L_b^{q),\theta_2}}. 
\end{align}
\end{definition}
Now let's give some examples to show that this space is not an empty set.
\begin{example}
 Assume that $a(x),b(x)\in L^1(\mathbb R^n).$ Then for   $ 0<\varepsilon \leq p-1, \  p>1,$ and $ 0<\eta\leq q-1, \  q>1$ we have $L^{p}\left( \mathbb R^{n} \right) \subset L_a^{p),\theta_1}( \mathbb R^{n}) $ and $L^{q}\left(\mathbb R^{n} \right) \subset L_b^{q),\theta_2}(\widehat R^{n}). $ Since  $\mathcal{S}\left(  R^{n}\right)\subset L^{p}\left(  R^{n} \right) $ and $\mathcal{S}\left( \widehat{\mathbb R^{n}}\right)\subset L^q\left( \widehat{ \mathbb R^{n}}\right) ,$ and Fourier transforms induce an isomorphism between $\mathcal{S}\left(\mathbb  R^{n}\right) $ and $\mathcal{S}\left( \widehat{\mathbb R^{n}}\right) ,$ then $\mathcal{S}\left(\mathbb  R^{n}\right)  \subset A_{q),\theta_2}^{p),\theta_1}\left(\mathbb R^{n}\right),$ where $\mathcal{S}\left(  \mathbb R^{n}\right) $ and $\mathcal{S}\left( \widehat{\mathbb R^{n}}\right) $ are  Schwartz space over $\mathbb R^{n}$ and $\widehat{\mathbb R^n},$ respectively.
\end{example}
\begin{example}
 Let $a(x),b(x)\in L^1(\mathbb R^n).$  Then  from  Example 1 we obtain $ C^{\infty}_0( \mathbb R^n)\subset \mathcal{S}\left(\mathbb  R^{n}\right)  \subset A_{q),\theta_2}^{p),\theta_1}\left( R^{n}\right).$ 
\end{example}

\begin{theorem}
Let $1<p,q<\infty,\ \theta_1, \theta_2>0\  \text{and}\  a(x), b(x)\ weight functions $ on $ R^{n}$. Then the norm of  $A_{q),\theta_2}^{p),\theta_1}\left( R^{n}\right)$ satisfies the following properties, where $f,g$ and $f_{n
\text{ }}$are in $A_{q),\theta_2}^{p),\theta_1}\left( R^{n}\right),$ $\lambda
\geq 0$ and $E$ is a measurable subset of $ \mathbb R^{n}:$
\end{theorem}

$1.\left\Vert f\right\Vert _{A_{q),\theta_2}^{p),\theta_1}}\geq 0;$

$2.\left\Vert f\right\Vert _{A_{q),\theta_2}^{p),\theta_1}}=0 $ if and only if $f=0$ a.e in $\mathbb  R^{n}; $

$3.\left\Vert\lambda f\right\Vert _{A_{q),\theta_2}^{p),\theta_1}}=\lambda\left\Vert f\right\Vert _{A_{q),\theta_2}^{p),\theta_1}};$

$4.\left\Vert f+g\right\Vert _{A_{q),\theta_2}^{p),\theta_1}}\leq \left\Vert f\right\Vert _{A_{q),\theta_2}^{p),\theta_1}}+\left\Vert g\right\Vert _{A_{q),\theta_2}^{p),\theta_1}};$


$5.$ If $0\leq f_{n}\uparrow f$ a.e. in $ \mathbb R^{n} ,$ then $\left\Vert
f_{n}\right\Vert _{A_{q),\theta_2}^{p),\theta_1}}\uparrow
\left\Vert f\right\Vert _{A_{q),\theta_2}^{p),\theta_1}};$


6. Let $1\leq a(x)<\infty $ be a weight function and let $ E\subset\mathbb R^{n}$ be a compact subset. Then
\begin{align}
 \int\limits_{E}\left\vert f\right\vert dx\leq C_0\left( p,E\right)
\left\Vert f\right\Vert _{A_{q),\theta_2}^{p),\theta_1}}
\end{align}
 for some $0<C_0<\infty .$

The first four properties follow from the definition of the norm  
$\left\Vert .\right\Vert_{A_{q),\theta_2}^{p),\theta_1}},$ and the corresponding properties of the generalized grand Lebesgue space.
\begin{proof}

Proof of property $5.$

If $ 0\leq f_{n}\uparrow f$ a.e in $\mathbb  R^{n}$ then
\begin{align}
\notag \|f_n\|_{L_{a}^{p),\theta_1}}\uparrow\sup_{n}\|f_n\|_{L_{a}^{p),\theta_1}}&= \sup_{n}( \sup_{0<\varepsilon\leq p-1}\varepsilon^{\theta_1}\|f_n\|_{L^{p-\varepsilon}( R^{n}, a^{\frac{\varepsilon}p})})\\&=\sup_{0<\varepsilon\leq p-1}\varepsilon^{\theta_1}( \sup_{n}\|f_n\|_{L^{p-\notag\varepsilon}( R^{n}, a^{\frac{\varepsilon}p})})\\&=\sup_{0<\varepsilon\leq p-1}\varepsilon^{\theta_1}( \|f\|_{L^{p-\varepsilon}( R^{n}, a^{\frac{\varepsilon}p})})=\|f\|_{L_{a}^{p),\theta_1}( R^{n}).}  
\end{align}  
Since $f_n \uparrow f$ a.e in $ R^{n},$ then$\hat f_n \uparrow\hat f $ in $\mathbb R^{n}.$ Similar to the expression (3.3)  we have
\begin{align}
\|\hat f_n\|_{L_{b}^{q),\theta_2}}\uparrow\sup_{n}\|\hat f_n\|_{L_{b}^{q),\theta_2}}&=\|\hat f\|_{L_{b}^{q),\theta_2}( \mathbb R^{n}).}
\end{align}
Thus from (3.3) and (3.4 we observe that
\begin{align*}
\|f_n\|_{A_{q),\theta_2}^{p),\theta_1}}&=\|{f_n}\|_{ L_{a}^{p)\theta_1}}+\|\hat f_n\|_{L_{b}^{q),\theta_2}}\uparrow\|f\|_{A_{q),\theta_2}^{p),\theta_1}}.
\end{align*}

Proof of property {6}:

 Let  $f\in A_{q),\theta_2}^{p),\theta_1}\left(\mathbb R^{n}\right)$ and $0<\varepsilon\leq p-1.$ Then $ f \in
L^{p),\theta_1}_a(\mathbb R^{n}) \ \text{and}\ \hat{f}\in L^{q),\theta_2}_b (\widehat{  \mathbb R^{n}}).$   By (2.10) we write $f\in{L^{p-\varepsilon }(\mathbb  R^{n}, a^{\frac{\varepsilon}p})}.$ Thus $f.a^{\frac{\varepsilon}{p(p-\varepsilon)}}\in\ L^{p-\varepsilon}(\mathbb R^{n})$ and $\chi_E.a^{-{\frac{\varepsilon}{p(p-\varepsilon)}}}\in L^{(p-\varepsilon)^{\prime}}( \mathbb R^{n}),$ where $\frac{1}{p-\varepsilon}+\frac{1}{(p-\varepsilon)^{'}}=1.$ Hence by Hölder's inequality we write
\begin{align}
\int_{E}\mid f\mid dx=&\int_{\mathbb R^{n}}\mid f\chi_{E}\mid dx=\int_{{\mathbb R^{n}}}\mid f.a^{\frac{\varepsilon}{p(p-\varepsilon)}}\mid \mid \chi_{E}a^{-{\frac{\varepsilon}{p(p-\varepsilon}}}\mid dx
\\& \leq\| f.a^{\frac{\varepsilon}{p(p-\varepsilon)}}\|_{L^{p-\varepsilon}}\|\chi_{E}a^{-{\frac{\varepsilon}{p(p-\varepsilon}}}\|\notag _{L^{(p-\varepsilon)^{'}}}\\&< C_1(\varepsilon,E) \|f\|_{L^{p-\notag\varepsilon}(\mathbb R^{n}, a^{\frac{\varepsilon}p})}< C_0\|{f}\|_{ L_{a}^{p)\theta_1}}\\&<C_0\left\Vert f\right\Vert _{A_{q),\theta_2}^{p),\theta_1}},
\end{align}  
where $C_0=\sup_{0<\varepsilon\leq p-1}\left\{C_1(E,\varepsilon)\right\}.$ This completes the proof.
\end{proof}
\begin{lemma} Let $1\leq p\leq\infty$ and let $a(x)$ be Beurling's weight function. For $f\in L^{p),\theta}_a\left(  \mathbb R^{n}\right)$ and $h\in\mathcal{S}(\mathbb R^{n})$ the integral 
\begin{equation}
<f,g>=\int_{ \mathbb R^{n}}f(x)g(x)dx
\end{equation}
converges, and $(f,g)\rightarrow <f,g>$ is a continuous bilinear form on $ L^{p),\theta}_a\left(  \mathbb R^{n}\right)X\mathcal{S}(\mathbb R^{n}).$ Hence $f\rightarrow <f,.>$ is a continuous linear embedding $ L^{p),\theta}_a\left(  \mathbb R^{n}\right)\hookrightarrow\mathcal{S}^{\prime }(\mathbb R^{n}).$
\end{lemma} 
\begin{proof}
Since  $ L^{p),\theta}_a\left(  \mathbb R^{n}\right)\hookrightarrow{L^{p-\varepsilon }(\mathbb  R^{n}, a^{\frac{\varepsilon}p})}$ and $a(x)\ge1,$ then $ L^{p),\theta}_a\left(  \mathbb R^{n}\right)\hookrightarrow{L^{p-\varepsilon }(\mathbb  R^{n} )}.$ Then it is easy to prove that the bilinear form $(f,g)\rightarrow <f,g>$ on $ L^{p),\theta}_a\left(  \mathbb R^{n}\right)X\mathcal{S}(\mathbb R^{n})$ is continuous. Thus  $ L^{p),\theta}_a\left(  \mathbb R^{n}\right)\hookrightarrow\mathcal{S}^{\prime }(\mathbb R^{n}).$
\end{proof}
\begin{remark}
Lemma 1 plays important role in the generalization of the classical Fourier transform to the spaces  $ L^{p),\theta}_a\left(  \mathbb R^{n}\right)$ in the following manner: It is known from Lemma 1 that  $ L^{p),\theta}_a\left(  \mathbb R^{n}\right)\hookrightarrow\mathcal{S}^{\prime }(\mathbb R^{n}).$ We also know that  generalized Fourier transform yields isometry between $ \mathcal{S}^{\prime }(\mathbb R^{n})$ and $ \mathcal{S}^{\prime }(\hat{\mathbb R^{n}})=\mathcal{S}^{\prime }(\mathbb R^{n}).$ Now we  define the Fourier transform 
\begin{align}
 \mathcal{F}:L^{p),\theta}_a\left(  \mathbb R^{n}\right)\hookrightarrow\mathcal{S}^{\prime }(\hat{\mathbb R^{n}})
\end{align}
as the composition of the map of Lemma 1, with the generalized Fourier transform 
\begin{align}
 \mathcal{F}:\mathcal{S}^{\prime }(\mathbb R^{n})\hookrightarrow\mathcal{S}^{\prime }(\hat{\mathbb R^{n}}).
\end{align}
\end{remark}

\begin{theorem}
Let $a(x)$ be Beurling's weight function. Then the normed space $A_{q),\theta_2}^{p),\theta_1}\left(\mathbb R^{n}\right)$ is a
Banach space.
\end{theorem}
\begin{proof}
Let $\left( f_{n}\right) _{n\in\mathbb{N}}$ be a Cauchy sequence in $A_{q),\theta_2}^{p),\theta_1}\left( R^{n}\right).$ Then $\left( f_{n}\right) _{n\in \mathbb{N}}$ is a Cauchy sequence in $L^{p),\theta_1}_a\left(  R^{n}\right)  $ and $(\hat{f}_{n})_{n\in \mathbb{N}}$ is a Cauchy sequence in $ L^{q),\theta_2}_b(\widehat{\mathbb  R^{n}}).$ Since  $L^{p),\theta_1}_a\left( \mathbb R^{n}\right)  $ and  $ L^{q),\theta_2}_b(\widehat{\mathbb  R^{n}})$ are Banach spaces (see \cite{u},\cite{su3}), there exist $f\in L^{p),\theta_1}_a\left(  \mathbb R^{n}\right)  $ and
$h\in  L^{q),\theta_2}_b(\widehat{ \mathbb R^{n}}) $ such that $f_{n}\rightarrow f\ $\ in $ L^{p),\theta_1}_a\left( \mathbb R^{n}\right)    $ and $\hat{f}_{n}\rightarrow h$ in $ L^{q),\theta2}_b(\widehat{\mathbb  R^{n}}).$ Then for given $
\eta >0,$ there exists $n_{0}\in
\mathbb{N}
,$ such that
\begin{equation}
\left\Vert  f_n-f\right\Vert _{L_a^{p),\theta_1}}<\frac{\eta }{2}\ \text{ and }\
\left\Vert \hat{f}_{n}-h\right\Vert _{L_b^{q),\theta_2}}<\frac{\eta }{2}
\end{equation}
for all $n>n_{0}.$  But by $(3.8)$ and $3.9$ in Remark 1, the Fourier transform  
 $\hat{f}\rightarrow \hat{f}$ 
is continuous from  $L^{p),\theta}_a\left(  \mathbb R^{n}\right)$ into $ \mathcal{S}^{\prime }(\hat{\mathbb R^{n}}).$  Thus\ $\hat{f}$ $=h,$ and we have
\begin{eqnarray}
\left\Vert f_{n}-f\right\Vert _{A_{q),\theta_2}^{p),\theta_1} }
&=&\left\Vert f_{n}-f\right\Vert _{L_a^{p),\theta_1}}+\left\Vert \widehat{f_{n}}
-\widehat{f}\right\Vert _{L_a^{q),\theta2}} \\
\notag&=&\left\Vert f_{n}-f\right\Vert _{L_a^{p),\theta_1}}+\left\Vert \widehat{f_{n}}
-{h}\right\Vert _{L_a^{q),\theta_2}}<\frac{\eta }{2}+\frac{\eta
}{2}=\eta 
\end{eqnarray}
for all $n\geq n_{0},$ and so  $A_{q),\theta_2}^{p),\theta_1}\left( R^{n}\right)$ is a Banach space.

\end{proof}

\begin{proposition}
$1.$ Let $a(x)$ be a Beurling weight function. Then the space  $A_{q),\theta_2}^{p),\theta_1}\left(\mathbb R^{n}\right)$ is translation invariant.  

$2.$ The space  $ A_{q),\theta_2}^{p),\theta_1}\left(\mathbb R^{n}\right) $ is modulation invariant, i.e if  $f\in A_{q),\theta_2}^{p),\theta_1}\left(\mathbb R^{n}\right) ,$ then $M_\xi f\in  A_{q),\theta_2}^{p),\theta_1}\left( \mathbb R^{n}\right) .$
\end{proposition}
\begin{proof}
Since $a(x)$ is submultiplicative and for any  $f\in A_{q),\theta_2}^{p),\theta_1}\left(\mathbb R^{n}\right) $  and $y_0\in\mathbb R^{n},$  we have $\widehat {T_{y_0} f}=e^{i <y_0,x>}\widehat f(x),$  a simple calculation shows that $ A_{q),\theta_2}^{p),\theta_1}\left(\mathbb R^{n}\right) $ is translation invariant.
Also since  $\widehat M_{\xi} f= T_{\xi} \hat f$  and $L_b^{q),\theta_2 }\left( \mathbb R^n\right) $ is translation invariant, it is easy to see that $A_{q),\theta_2}^{p),\theta_1}\left(\mathbb R^{n}\right)$ is modulation invariant.
\end{proof}

\begin{remark}
We denote by $\overline{C_0^{\infty}}(\mathbb R^{n})|_{L_a^{p),\theta }\left(\mathbb R^{n}\right) }$ and $\overline{C_0^{\infty}}( \mathbb R^{n})|_{A_{q),\theta _{2}}^{p),\theta _{1}}\left(\mathbb R^{n}\right) },$  the closure of the set ${C_0^{\infty}}(\mathbb R^{n})$ in the space $ {L_a^{p),\theta }\left(\mathbb R^{n}\right) }$ and   $A_{q),\theta_2}^{p),\theta_1}\left(\mathbb R^{n}\right)$ respectively. If $a\in L^1( R^{n}),$ and $b\in L^1(\hat R^{n}),$ we know by Example 1 and Example 2 that $\mathcal{S}\left(  \mathbb R^{n}\right)  \subset A_{q),\theta_2}^{p),\theta_1}\left(\mathbb R^{n}\right)$ and $ C^{\infty}_0(\mathbb R^{n})\subset \mathcal{S}\left( \mathbb R^{n}\right)  \subset A_{q),\theta_2}^{p),\theta_1}\left( \mathbb R^{n}\right).$ 
\end{remark}

\begin{theorem}
Let $1< p,q< \infty$ and $a(x),b(x)\in L^1(\mathbb R^{n}).$ Then the closure set   $\overline{C_0^{\infty}}( \mathbb R^{n})|_{A_{q),\theta_2}^{p),\theta_1}\left( \mathbb R^{n}\right) }$  of the set ${C_0^{\infty}}(\Omega)$ in the space  $ A_{q),\theta_2}^{p),\theta_1}\left( \mathbb R^{n}\right)$ consists of $ f\in  A_{q),\theta_2}^{p),\theta_1}\left( \mathbb R^{n}\right), $ such that 
\begin{align}
lim_{\varepsilon\to 0} \varepsilon^{\theta_1} \int\limits_{\mathbb R^{n}}\left\vert f\left(t\right)\right\vert ^{p-\varepsilon } a^{\frac{\varepsilon}p}(t)dt =0. 
\end{align} 
and 
\begin{align}
lim_{\varepsilon\to 0} \varepsilon^{\theta_2} \int\limits_{\mathbb R^{n}}\left\vert\hat f\left(t\right)\right\vert ^{p-\varepsilon } b^{\frac{\varepsilon}p}(t)dt =0. 
\end{align}
\end{theorem}
\begin{proof}
 Let  $f\in\overline{C_0^{\infty}}(\mathbb R^{n})|_ { A_{q),\theta_2}^{p),\theta_1}}.$  Then there exists a sequence $\left( f_{n}\right)$ in  $ C_0^{\infty}(\mathbb R^{n})$ such that 
\begin{equation*}
\|f_n-f\|_{  A_{q),\theta_2}^{p),\theta_1}}\rightarrow 0.
\end{equation*}
This implies  $\left( f_{n}\right)$ converges to $f$ in $ {L_a^{p),\theta_1 }\left(\mathbb R^{n}\right) }$ so $f\in \overline{C_0^{\infty}}(\mathbb R^{n})|_{L_a^{p),\theta }\left(\mathbb R^{n}\right) }, $ and  $\left(\hat f_{n}\right)$ converges to $\hat f$ in $ {L_b^{q),\theta_2 }\left(\mathbb R^{n}\right) }.$ Since $f\in \overline{C_0^{\infty}}(\mathbb R^{n})|_{L_a^{p),\theta }\left(\mathbb R^{n}\right) }, $ we have
\begin{align}
lim_{\varepsilon\to 0} \varepsilon^{\theta} \int\limits_{\mathbb R^{n}}\left\vert f\left(t\right)\right\vert ^{p-\varepsilon } a^{\frac{\varepsilon}p}(t)dt =0. 
\end{align} 
(see \cite{su3}).

Since $\left(\hat f_{n}\right)$ converges to $\hat f$ in $\mathcal {F}( C_0^{\infty}(\mathbb R^n))\subset{L_b^{q),\theta_2 }\left(\mathbb R^{n}\right) }$, for given $\delta >0,$ there exists $n_{0}\in\mathbb{N} $ such that 
\begin{equation}
\|\hat f_{n_0}-\hat f\|_{{L_a^{q),\theta }\left(\mathbb R^{n}\right) }}<\frac{\delta 
}{2}.  
\end{equation}
for all $n>n_0.$ 
Take  the element $\hat f_{n_0}\in\mathcal {F}( C_0^{\infty}(\mathbb R^n))\subset\mathcal{S}(\mathbb R^n)).$
Let $r=\frac{q}{q-\varepsilon},$ and $ r'=\frac{q}{\varepsilon}.$ Since and $\frac{1}{r}+\frac{1}{r'}=1.$ Since
\begin{align*}
  \int\limits_{ \mathbb R^{n} }(\left\vert\hat f_{n_{0}}\left(t\right) \right\vert ^{q-\varepsilon })^{\frac{q}{q-\varepsilon }}dt = \int\limits_{\mathbb R^{n} }\left\vert\hat f_{n_{0}}\left(t\right) \right\vert ^{q}dt<\infty, 
\end{align*}
and 
\begin{align*}
 \int\limits_{\mathbb R^{n} }\left\vert b^{\frac{\varepsilon}q}(t)\right\vert^{\frac{q}{\varepsilon }}dt=  \int\limits_{\mathbb R^{n} }\left\vert b\right\vert dt=\|b(t)\|_{L^1}<\infty,
\end{align*}
 by the H\"{o}lder's inequality
\begin{align}
\int\limits_{\mathbb R^{n} }\left\vert\hat f_{n_{0}}\left(t\right) \right\vert ^{q-\varepsilon } b(t)^{\frac{\varepsilon}{q}}dt\leq&\| |\hat f_{n_{0}}|^{q-\varepsilon}\|_{L^{\frac{q}{q-\varepsilon}}}.\|b^{\frac{\varepsilon}{b}}\|_{L^{\frac{p}{\varepsilon}}}=(\|f_{n_0}\|_{L^q})^{q-\varepsilon}(\|a\|_{L^1})^{\frac{\varepsilon}{q}}.
\end{align}
 Since $b\in L^1 (\mathbb R^n)$ and $f_{n_0} \in  L^q (\mathbb R^n),$   from  (3.18) we write
 \begin{align}
\varepsilon^{\theta_2}( \int\limits_{\mathbb R^{n} }\left\vert\hat f_{n_{0}}\left(t\right) \right\vert ^{q-\varepsilon_2 } b(t)^{\frac{\varepsilon}{q}}dt)^{\frac{1}{q-\varepsilon}}\leq&\varepsilon^{\theta_2}((\|\hat f_{n_0}\|_{L^q})^{q-\varepsilon}(\|b\|_{L^1})^{\frac{\varepsilon}{q}})^{\frac{1}{q-\varepsilon}}.
 \end{align}
 This implies,
\begin{align}
\varepsilon^{\theta_2}(\int\limits_{\mathbb R^{n} }\left\vert \hat f_{n_{0}}\left(t\right)\right\vert ^{q-\varepsilon } a^{\frac{\varepsilon}q}(t)dt)^{\frac{1}{q-\varepsilon}} \rightarrow 0
\end{align}
as $\varepsilon$ tends to zero. Thus  there exists $\varepsilon _0>0$ such that when $\varepsilon <\varepsilon_0,$ we have 
\begin{align}
\varepsilon^{\theta_2}( \int\limits_{\mathbb R^{n}}\left\vert\hat f_{n_{0}}\left(t\right)\right\vert ^{q-\varepsilon } b^{\frac{\varepsilon}q}(t)dt)^{\frac{1}{q-\varepsilon}} &<\frac{\delta}{2}.
\end{align}
 From (3.17) and (3.22) we obtain 
  \begin{align*}
    \varepsilon^{\theta_2}( \int\limits_{ \mathbb R^{n}}\left\vert\hat f\left(t\right)\right\vert ^{q-\varepsilon } b^{\frac{\varepsilon}q}(t)dt)^{\frac{1}{q-\varepsilon}}&\leq \varepsilon^{\theta_2}( \int\limits_{\mathbb R^{n}}\left\vert\hat f\left(t\right)-\hat f_{n_0}(t)\right\vert ^{q-\varepsilon } b^{\frac{\varepsilon}q}(t)dt)^{\frac{1}{q-\varepsilon}}\\+ \varepsilon^{\theta_2}( \int\limits_{\mathbb R^{n}}\left\vert\hat f_{n_0}\left(t\right)\right\vert ^{q-\varepsilon } b^{\frac{\varepsilon}q}(t)dt)^{\frac{1}{q-\varepsilon}}&<\|f-\hat f_{n_0}\|_{{L_a^{q),\theta_2 }\left(\mathbb R^{n}\right) }}+\frac{\delta}2<
\delta 
   \end{align*}
   when  $\varepsilon<\varepsilon_0.$ This completes the proof.
   \end{proof}
     


\section{Inclusions for the space $ A_{q),\theta_2 }^{p),\theta _{1}}\left(\mathbb R^{n},a, b\right) $}

\begin{proposition}
 Let $a_1(x),a_2(x),b_1(x),b_2(x)$ be a Beurling's weight functions. Then  $$ A_{q),\theta }^{p),\theta }\left(\mathbb R^{n},a_1, b_1\right)\subset  A_{q),\theta }^{p),\theta }\left(\mathbb R^{n},a_2, b_2\right) $$ if and only if there exists $C>0$ such that
\begin{align}
\|f\|_{ A_{q),\theta }^{p),\theta }\left(\mathbb R^{n},a_2, b_2\right)}\leq C\|f\|_{ A_{q),\theta }^{p),\theta }\left(\mathbb R^{n},a_1, b_1\right).}
\end{align}
\end{proposition}
\begin{proof}
\end{proof}Suppose  $$ A_{q),\theta }^{p),\theta }\left(\mathbb R^{n},a_1, b_1\right)\subset  A_{q),\theta }^{p),\theta }\left(\mathbb R^{n},a_2, b_2\right). $$ Define the sum norm 
  \begin{align}
\||f|||=\|f\|_{ A_{q),\theta }^{p),\theta }\left(\mathbb R^{n},a_2, b_2\right)}+\|f\|_{ A_{q),\theta }^{p),\theta }\left(\mathbb R^{n},a_1, b_1\right).}
  \end{align}
 on $A_{q),\theta }^{p),\theta }\left(\mathbb R^{n},a_1, b_1\right).$ Let $\left( f_{n}\right) _{n\in\mathbb{N}}$ be a Cauchy sequence in $(A_{q),\theta }^{p),\theta }\left(\mathbb R^{n},a_1, b_1\right),\||.\||).$  Then $\left( f_{n}\right) _{n\in \mathbb{N}}$ is a Cauchy sequence in  $A_{q),\theta }^{p),\theta }\left(\mathbb R^{n},a_1, b_1\right)$ and $(\hat{f}_{n})_{n\in \mathbb{N}}$ is a Cauchy sequence in $ A_{q),\theta }^{p),\theta }\left(\mathbb R^{n},a_2, b_2\right).$ Hence $\left( f_{n}\right) _{n\in \mathbb{N}}$ coverges to a function $f$ in $A_{q),\theta }^{p),\theta }\left(\mathbb R^{n},a_1, b_1\right)$ and to a function $g$ in  $ A_{q),\theta }^{p),\theta }\left(\mathbb R^{n},a_2, b_2\right).$ It is easy to show that $f=g,$ and so $(A_{q),\theta }^{p),\theta }\left(\mathbb R^{n},a_1, b_1\right),\||.\||)$ is complete. This shows that the original norm of $A_{q),\theta }^{p),\theta }\left(\mathbb R^{n},a_1, b_1\right)$ and $\||.\||$ are equivalent. Then there exist $C_1>0, C_2>0$ such that 
\begin{align}
C_2\|f\|_{ A_{q),\theta }^{p),\theta }\left(\mathbb R^{n},a_2, b_2\right)}\leq\||f|||\leq C_1\|f\|_{ A_{q),\theta }^{p),\theta }\left(\mathbb R^{n},a_1, b_1\right).}
\end{align}
Hence if one uses the closed graph theorem obtains 
\begin{align}
\|f\|_{ A_{q),\theta }^{p),\theta }\left(\mathbb R^{n},a_2, b_2\right)}\leq C\|f\|_{ A_{q),\theta }^{p),\theta }\left(\mathbb R^{n},a_1, b_1\right),}
\end{align}
where $C=max\{C_1,C_2\}.$
To prove of the other direction is easy.

It is easy to prove the following Proposition.
\begin{proposition}

$1)$ If   $\theta _{1}<\theta _{3},$  and  $\theta _{2}<\theta _{4},$ then we have the inclusion
\begin{equation*}
A_{q),\theta_2 }^{p),\theta _{1}}\left( \mathbb R^{n},a\right) \subset A_{q),\theta_4
}^{p),\theta _{3}}\left(\mathbb R^{n},a\right).
\end{equation*}

$2)$ Let  $w_1(x)=a_1(x)^{\frac{\varepsilon}{p}}, \ w_2(x)=b_1(x)^{\frac{\eta}{q}},w_{3}(x)=a_2(x)^{\frac{\varepsilon}{p}}$ and $w_4(x)=b_2(x)^{\frac{\eta}{q}},$ for $0<\varepsilon\leq p-1, 0<\eta\leq q-1.$ If $w_1(x)\prec w_3(x),$ and  $w_2(x)\prec w_4(x),$ then
\begin{equation*}
A_{q),\theta}^{p),\theta}\left(\mathbb R^{n},a_2,b_2\right)  \subset A_{q),\theta
}^{p),\theta }\left( \mathbb R^{n},a_1,b_1\right).
\end{equation*}

\end{proposition}
By using Proposition 3, the following corollary is easily proved.
\begin{corollary}
Let $\theta_1\leq\theta_3,\ \theta_2\leq\theta_4$ and let  $w_1(x)=a_1(x)^{\frac{\varepsilon}{p}},\ w_2(x)=b_1(x)^{\frac{\eta}{q}},\ w_{3}(x)=a_2(x)^{\frac{\varepsilon}{p}} ,\ w_4(x)=b_2(x)^{\frac{\eta}{q}},$ for $0<\varepsilon\leq p-1, 0<\eta\leq q-1.$ If $w_1(x)\prec w_3(x),$ and  $w_2(x)\prec w_4(x),$ then
\begin{equation*}
A_{q),\theta_2}^{p),\theta_1}\left(\mathbb R^{n},a_2,b_2\right)  \subset A_{q),\theta_4
}^{p),\theta_3 }\left( \mathbb R^{n},a_1,b_1\right).
\end{equation*}
\end{corollary}
\ We know that when $\theta _{1}=0$, $\theta _{2}=0,$ and a=b=1 the spaces $ L_a^{p),\theta_1}(\mathbb R^{n}) $ and $ L_b^{q),\theta_2}(\widehat{\mathbb R^{n}}) $ reduce to $L^{p}\left(\mathbb  R^{n}\right) $ and $L ^{q}\left(
\widehat{\mathbb R^{n}}\right) $ respectively. Then $A_{q),\theta _{2}}^{p),\theta
_{1}}\left( \mathbb R^{n}\right) $ reduces to the space $A^{p,q}\left(\mathbb  R^{n}\right) ,$which
is defined by the norm $\left\Vert f\right\Vert _{A^{p,q}}=\left\Vert
f\right\Vert _{p}+\left\Vert \widehat{f}\right\Vert _{L ^{q}},$ ( see \cite{rgl}, \cite{wr}).
\begin{definition}
 Let's denote the space defined by the norm  
\begin{align*}
\left\Vert f\right\Vert _{ A^{p-\varepsilon , q-\eta  }(\mathbb R^n,a(x)^{\frac{\varepsilon}{p}}, b(x)^{\frac{\eta}{q}})}&=\left\Vert f\right\Vert
_{L^{p-\varepsilon }(\mathbb R^{n}, a(x)^{\frac{\varepsilon}{p}})}+\left\Vert \widehat{f}\right\Vert _{L ^{q-\eta }(\mathbb R^n, b(x)^{\frac{\eta}{q}})} 
\end{align*}
 with $ { A^{p-\varepsilon , q-\eta  }(\mathbb R^n,a(x)^{\frac{\varepsilon}{p}}, b(x)^{\frac{\eta}{q}})}.$ 
\end{definition}
\begin{theorem}
 Let $1<p,q<\infty $ and $a(x),b(x)\in L^1( R^{n}) .$ We have the inclusions
\begin{align*}
A^{p,q}\left(\mathbb R^{n}\right) \subset A_{q),\theta _{2}}^{p),\theta _{1}}\left( \mathbb R^{n}\right) \subset{ A^{p-\varepsilon , q-\eta  }(\mathbb R^n,a(x)^{\frac{\varepsilon}{p}}, b(x)^{\frac{\eta}{q}}).}
\end{align*}
\end{theorem}

\begin{proof}
{\rm (a)} Let $f\in A_{q),\theta _{2}}^{p),\theta _{1}}\left(\mathbb R^{n}\right) .$ Then $f\in
L^{p),\theta _{1}}_a\left(\mathbb R^{n}\right) $ and $\ \widehat{f}\in L ^{q),\theta
_{2}}_b(\widehat{\mathbb R^{n}}).$ Since by (2.12),  $L^{p),\theta _{1}}_a\left(\mathbb R^{n}\right)
\subset L^{p-\varepsilon }\left(\mathbb  R^{n},a(x)^{\frac{\varepsilon}{p}}\right) $ and $L ^{q),\theta _{2}}(
\widehat{\mathbb R^{n}})\subset L^{q-\eta }(\widehat{\mathbb R^{n}},b_1(x)^{\frac{\eta}{q}}),$ there
exists constants $C_{1}>0$ and $C_{2}>0$ such that $\left\Vert f\right\Vert
_{L^{p-\varepsilon }(\mathbb R^{n}, a(x)^{\frac{\varepsilon}{p}})}\leq C_{1}\left\Vert f\right\Vert _{L_a^{p),\theta _{1}}}$ and $\left\Vert \widehat{f}\right\Vert _{L ^{q-\eta }(\mathbb R^n, b(x)^{\frac{\eta}{q}})}\leq
C_{2}\left\Vert \widehat{f}\right\Vert _{L_b^{q),\theta _{2}}}.$ Thus
\begin{align*}
\left\Vert f\right\Vert _{ A^{p-\varepsilon , q-\eta  }(\mathbb R^n,a(x)^{\frac{\varepsilon}{p}}, b(x)^{\frac{\eta}{q}})}&=\left\Vert f\right\Vert
_{L^{p-\varepsilon }(\mathbb R^{n}, a(x)^{\frac{\varepsilon}{p}})}+\left\Vert \widehat{f}\right\Vert _{L ^{q-\eta }(\mathbb R^n, b(x)^{\frac{\eta}{q}})}\\ &\leq   C\left\{ \left\Vert f\right\Vert _{L_a^{p),\theta _{1}}}+\left\Vert \widehat{f}\right\Vert _{L_b^{q),\theta _{2}}}\right\} =C\left\Vert f\right\Vert _{A_{q),\theta
\notag_{2}}^{p),\theta _{1}}},
\end{align*}
where $C=\max \left\{ C_{1},C_{2}\right\} .$Then $f\in { A^{p-\varepsilon , q-\eta  }(\mathbb R^n,a(x)^{\frac{\varepsilon}{p}}, b(x)^{\frac{\eta}{q}})} $ and so,
\begin{equation*}
A_{q),\theta _{2}}^{p),\theta _{1}}\left( R^{n}\right) \subset \
{ A^{p-\varepsilon , q-\eta  }(\mathbb R^n,a(x)^{\frac{\varepsilon}{p}}, b(x)^{\frac{\eta}{q}})}.
\end{equation*}

Now let $g\in A^{p,q}\left(\mathbb R^{n}\right) .$ Then $g\in L^{p}\left(\mathbb  R^{n}\right) $
and $\widehat{g}\in L ^{q}(\widehat{ \mathbb R^{n}}).$ Since $L^{p}\left(\mathbb R^{n}\right)
\subset L^{p),\theta _{1}}\left( \mathbb R^{n}\right) $ and $ L ^{q}(\widehat{\mathbb R^{n}})\subset  L ^{q),\theta
_{2}}_b(\widehat{\mathbb R^{n}}),$ there exists $D_{1}>0$ and $
D_{2}>0$ such that $\left\Vert g\right\Vert _{L_a^{p),\theta _{1}}}\leq
D_{1}\left\Vert g\right\Vert _{L^{p}}$ and $\left\Vert \widehat{g}\right\Vert
_{L_b ^{q),\theta _{2}}}\leq D_{2}\left\Vert \widehat{g}\right\Vert _{L^{q}}.$ Then
\begin{eqnarray*}
\left\Vert g\right\Vert _{A_{q),\theta _{2}}^{p),\theta _{1}}} \leq D_{1}\left\Vert g\right\Vert _{L^p}+D_{2}\left\Vert
\widehat{g}\right\Vert _{L ^{q}} &\leq &D\left( \left\Vert g\right\Vert _{L^{p}}+\left\Vert \widehat{g}
\notag\right\Vert _{L ^{q}}\right) =D\left\Vert g\right\Vert _{A^{p,\text{ }q}},
\end{eqnarray*}
where $D=\max \left\{ D_{1},D_{2}\right\} .$ This implies $g\in A_{q),\theta
_{2}}^{p),\theta _{1}}\left(  R^{n}\right) $ and so,
\begin{equation}
A^{p,q}\left(\mathbb  R^{n}\right) \subset A_{q),\theta }^{p),\theta }\left(\mathbb R^{n}\right) .
\notag
\end{equation}
\end{proof}
\begin{theorem}
 Let $1<p,q<\infty $ and $a(x),b(x)\in L^1(\mathbb R^{n}) .$

{\rm (a)}
 $\ {C_{0}^{\infty }\left(\mathbb R^{n}\right) }$ is dense in $ {A^{p,q}\left(\mathbb  R^{n}\right) },$ i.e 

\begin{equation*}
\ \overline{C_{0}^{\infty }\left( \mathbb R^{n}\right) }\mid _{A^{p,q}}=A^{p,q}\left(
 R^{n}\right) ,
\end{equation*}%
where $\ \overline{C_{0}^{\infty }\left(  R^{n}\right) }\mid _{A^{p,q}}$is the
closure of the set $C_{0}^{\infty }\left( R^{n}\right) $ in the space $%
A^{p,q}\left( R^{n}\right) .$

{\rm (b)}
\begin{equation*}
\overline{A^{p,q}\left( \mathbb R^{n}\right) }\mid _{A_{q),\theta _{2}}^{p),\theta_{1}}}
=\overline{C_{0}^{\infty }\left( \mathbb R^{n}\right) }\mid _{A_{q),\theta_{2}}^{p),\theta _{1}}},
\end{equation*}
where $\overline{A^{p,q}\left( \mathbb R^{n}\right) }\mid _{A_{q),\theta _{2}}^{p),\theta_{1}}}$ is the closure of the set ${A^{p,q}\left(\mathbb  R^{n}\right) }$ 
in the space $A_{q),\theta _{2}}^{p),\theta _{1}}\left(\mathbb R^{n}\right) .$

\begin{proof} 
Using the continuity of the Fourier transform, the property $( a)$ is easily proved

(b). By the inclusion
\begin{equation*}
C_{0}^{\infty }\left( \mathbb R^{n}\right) \subset{A^{p,q}\left(\mathbb  R^{n}\right) }\subset
A_{q),\theta _{2}}^{p),\theta _{1}}\left(\mathbb R^{n}\right) ,
\end{equation*}
we have
\begin{equation}
\overline{C_{0}^{\infty }\left( \mathbb R^{n}\right) }\mid _{A_{q),\theta
_{2}}^{p),\theta _{1}}}\subset \overline{{A^{p,q}\left(\mathbb  R^{n}\right) }}
\mid_{A_{q),\theta _{2}}^{p),\theta _{1}}}.  
\end{equation}
Since $\left\Vert f\right\Vert _{A_{q),\theta
_{2}}^{p),\theta _{1}}}\leq \left\Vert f\right\Vert _{A^{p,q}},$ this implies 
\begin{equation}
\ \overline{C_{0}^{\infty }\left(\mathbb R^{n}\right) }\mid _{A^{p,q}}\subset \overline{
C_{0}^{\infty }\left( \mathbb R^n\right) }\mid _{A_{q),\theta _{2}}^{p),\theta _{1}}}.
\end{equation}
From $\left( 4.8\right) $ and by (a), we have
\begin{equation*}
\overline{C_{0}^{\infty }\left(\mathbb R^{n}\right) }\mid _{A^{p,q}}=A^{p,q}\left(
\mathbb R^{n}\right) \subset \overline{C_{0}^{\infty }\left(\mathbb R^{n}\right) }\mid
_{A_{q),\theta _{2}}^{p),\theta _{1}}}.
\end{equation*}
This inclusion implies
\begin{equation}
\overline{A^{p,q}\left(\mathbb  R^{n}\right) }\mid _{A_{q),\theta _{2}}^{p),\theta_{1}}}\subset \overline{C_{0}^{\infty }\left(  \mathbb R^{n}\right) }\mid _{A_{q),\theta
_{2}}^{p),\theta _{1}}}.  
\end{equation}
Combining (4.7) and (4.9), we obtain
\begin{equation*}
\overline{A^{p,q}\left( \mathbb R^{n}\right) }\mid _{A_{q),\theta _{2}}^{p),\theta_{1}}}
=\overline{C_{0}^{\infty }\left(\mathbb R^{n}\right) }\mid _{A_{q),\theta_{2}}^{p),\theta _{1}}}.
\end{equation*}
\end{proof}

\end{theorem}

\section{The multipliers from  ${L^{1 }(\mathbb R^{n}, a^{\frac{\varepsilon}{p}})}$  to  $( L_a^{p),\theta}\left(\mathbb R^{n}\right))^{\ast}$ and from ${L^{1 }(\mathbb R^{n}, a^{\frac{\varepsilon}{p}})}$    to  $(A_{q),\theta_2}^{p),\theta_1}\left(\mathbb R^{n}\right))^{\ast}$  for ${0<\varepsilon \leq p-1}$ and   duality}  

In this section we will investigate the multipliers from  ${L^{1 }(\mathbb R^{n}, a^{\frac{\varepsilon}{p}})}$  into $ L_a^{p),\theta}\left(\mathbb R^{n}\right)$ and  ${L^{1 }(\mathbb R^{n}, a^{\frac{\varepsilon}{p}})}$  into  $A_{q),\theta_2}^{p),\theta_1}\left(\mathbb R^{n}\right)$  for ${0<\varepsilon \leq p-1}$.  First, we need the following Theorem. 
\begin{theorem}
 Let a(x) be a Beurling's weight function. Then

 a) The generalized  grand Lebesgue space $ L_a^{p),\theta}\left(\mathbb R^{n}\right)$ is an essentially Banach convolution module  over ${L^{1 }(\mathbb R^{n}, a^{\frac{\varepsilon}{p}})}$ for ${0<\varepsilon \leq p-1}$.

 b) The space  $A_{q),\theta_2}^{p),\theta_1}\left(\mathbb R^{n}\right)$ is an essentially Banach convolution module over  ${L^{1 }(\mathbb R^{n}, a^{\frac{\varepsilon}{p}})}$ for ${0<\varepsilon \leq p-1}$.
\end{theorem}
\begin{proof}
a) It is easy to show that $ L_a^{p),\theta}\left(\mathbb R^{n}\right)$   is a module on  ${L^{1 }(\mathbb R^{n}, a^{\frac{\varepsilon}{p(p-\varepsilon)}})}$ algebraically  for ${0<\varepsilon \leq p-1}$ .
Now let $ f\in{L^{1 }(\mathbb R^{n}, a^{\frac{\varepsilon}{p(p-\varepsilon)}})}$ for ${0<\varepsilon \leq p-1}$ and $g\in  L_{a}^{p),\theta}(\mathbb R^n).$ Then
\begin{align}
\left\Vert f\ast g\right\Vert _{L_a^{p),\theta}}&=\sup_{0<\varepsilon \leq p-1}\varepsilon^{\theta_1} \left\Vert f\ast g\right\Vert_{L^{p-\varepsilon }(\mathbb R^{n}, a^{\frac{\varepsilon}p})} = \sup_{0<\varepsilon \leq p-1}\varepsilon^{\theta} \left\Vert \int_{\mathbb R^{n}}f(t) g(x-t)dt \right\Vert_{L^{p-\varepsilon }(\mathbb R^{n}, a^{\frac{\varepsilon}p})}\notag
\\&\leq\sup_{0< \varepsilon \leq  p-1}\varepsilon^{\theta}\int_{\mathbb R^{n}}\mid f(t)\mid\|g(x-t)\|_{{L^{p-\varepsilon }(\mathbb R^{n}, a^{\frac{\varepsilon}p)}}}dt
\end{align}

If we make the substitution $u=x-t$ in (6.1), we obtain
\begin{align}
\left\Vert f\ast g\right\Vert _{L_a^{p),\theta}}&\leq \sup_{0< \varepsilon \leq  p-1}\varepsilon^{\theta}\|g\|_{{L^{p-\varepsilon }(\mathbb R^{n}, a^{\frac{\varepsilon}p})}}\int_{\mathbb R^{n}}\mid f(t)\mid a^{\frac{\varepsilon}{p(p-\varepsilon}} dt
\\
&=\left\Vert f\right\Vert _{{L^{1 }(\mathbb R^{n}, a^{\frac{\varepsilon}{p(p-\varepsilon}}})}\left\Vert
g\right\Vert _{L_a^{p),\theta_1}}\leq\left\Vert f\right\Vert _{{L^{1 }(\mathbb R^{n}, a^{\frac{\varepsilon}{p}}})}\left\Vert
g\right\Vert _{L_a^{p),\theta_1}}.\   \notag
\end{align}
Thus we proved that  $ L_a^{p),\theta}\left(\mathbb R^{n}\right)$ is a Banach convolution module  over ${L^{1 }(\mathbb R^{n}, a^{\frac{\varepsilon}{p}})}.$

Now let's show that $ L_a^{p),\theta}\left(\mathbb R^{n}\right)$ is essentially. Since a(x) is a Beurling's weight function, by Theorem $ 5.3 $ in \cite{w},  ${L^{1 }(\mathbb R^{n}, a^{\frac{\varepsilon}{p}})}$ is Banach convolution algebra and admits a bounded approximate $(e_{\alpha})_{\alpha\in I}\subset \ {C_{0}^{\infty }\left(\mathbb R^{n}\right) }.$ It is also known from Lemma $6$ in \cite{mu} that for Beurling's weight function , the weighted $L^p\left(\mathbb R^{n}\right) $ space admits an approximate identity, that is bounded in weighted $L^1\left(\mathbb R^{n}\right) $ space. From this results we observe that  $(e_{\alpha})_{\alpha\in I}$ is an approximate identity of ${{L^{p-\varepsilon }(\mathbb R^{n}, a^{\frac{\varepsilon}{p}}})}$ for ${0< \varepsilon \leq  p-1}.$ Let  $f\in  L_{a}^{p),\theta}(\mathbb R^n)$ and let $\eta>0$. According to the definition of supremum there exists  ${0< \varepsilon_0 \leq  p-1},$ such that 
\begin{align}
\left\Vert e_{\alpha}\ast f-f\right\Vert _{L_a^{p),\theta}}&=\varepsilon_0^{\theta} \left\Vert e_{\alpha}\ast f-f\right\Vert_{L^{p-\varepsilon_0 }(\mathbb R^{n}, a^{\frac{\varepsilon_0}p})} +\frac{\eta}2.
\end{align}
Since  $(e_{\alpha})_{\alpha\in I} $ is an approximate identity of ${L^{p-\varepsilon_0 }(\mathbb R^{n}, a^{\frac{\varepsilon_0}p})},$ then there exists $\alpha_0\in I$ such that 
\begin{align}
\left\Vert e_{\alpha}\ast f-f\right\Vert_{L^{p-\varepsilon_0 }(\mathbb R^{n}, a^{\frac{\varepsilon_0}p})}\leq\frac{\eta}{2\varepsilon^{\theta}_0}. 
\end{align}
for all $\alpha>\alpha_0.$ Hence from $(5.3)$ and $(5.4)$ we obtain 
\begin{align}
\left\Vert e_{\alpha}\ast f-f\right\Vert _{L_a^{p),\theta}}<\eta
\end{align}
for all $\alpha>\alpha_0$. This completes the proof of the first part. 

b)  Let $ f\in{L^{1 }(\mathbb R^{n}, a^{\frac{\varepsilon}{p}})} $ and $g\in A_{q),\theta_2}^{p,\theta_1}\left(\mathbb R^{n}\right) .$ Then   $g\in
L^{p),\theta }_a\left(\mathbb R^{n}\right) $ and $\ \widehat{g}\in L ^{q),\theta_{2}}_b(\widehat{\mathbb R^{n}}).$ From (5.2) in part (a) we write
\begin{align}
\left\Vert f\ast g\right\Vert _{L_a^{p),\theta}}&\leq\left\Vert f\right\Vert _{{L^{1 }(\mathbb R^{n}, a^{\frac{\varepsilon}{p}}})}\left\Vert
g\right\Vert _{L_a^{p),\theta_1}}.
\end{align}
Since $f\in{L^{1 }(\mathbb R^{n}, a^{\frac{\varepsilon}{p}})} \subset L^{1}\left( \mathbb R^{n}\right) ,$  the Fourier transform $\hat{f}$\ is bounded and $
\left\Vert \hat{f}\ \right\Vert _{\infty }\leq \left\Vert f\right\Vert _{1}\leq\left\Vert f\right\Vert _{{L^{1 }(\mathbb R^{n}, a^{\frac{\varepsilon}{p}}})}$
Thus
\begin{align}
\left\Vert \widehat{f\ast g}\right\Vert_{L^{q),\theta_2}_b}& =\left\Vert\hat{f}\ .\widehat{g}\ \right\Vert _{L^{q)}_b}=
\sup_{0<\varepsilon \leq q-1}\varepsilon^{\theta_2} \left\Vert\hat{f}\ .\widehat{g}\right\Vert_{L^{q-\varepsilon }(\mathbb R^{n},\notag b^{\frac{\varepsilon}q})} 
\\
&\leq\sup_{0<\varepsilon \leq p-1}\varepsilon^{\theta_2} \left\Vert \hat{f}\ \right\Vert _{\infty\notag
}\left\Vert \widehat{g}\right\Vert _{L^{q-\varepsilon }(\mathbb R^{n}, b^{\frac{\varepsilon}p})}
\\ &\leq\left\Vert f\right\Vert _{{L^{1 }(\mathbb R^{n}, a^{\frac{\varepsilon}{p}}})}\left\Vert\widehat{g}\right\Vert _{L^{q),\theta_2}_b}.
\end{align}

Combining $\left( 5.6\right) $ and $\left( 5.7\right) $ we obtain
\begin{align}
\left\Vert f\ast g\right\Vert _{A_{q),\theta_2}^{p),\theta1}\left(\mathbb R^{n}\right) } &=\left\Vert f\ast g\right\Vert _{L_a^{p),\theta_1}}+\left\Vert \widehat{f\ast g}\right\Vert_{L^{q),\theta_2}_b} 
\\
& \leq\left\Vert f\right\Vert _{{L^{1 }(\mathbb R^{n}, a^{\frac{\varepsilon}{p}}})}\left\Vert g\right\Vert _{L_a^{p),\theta_1}}+\left\Vert f\right\Vert _{{L^{1 }(\mathbb R^{n}, a^{\frac{\varepsilon}{p}}})}\left\Vert \widehat{g}\right\Vert _{L^{q),\theta_2}_b}.\notag
\\
&=\left\Vert f\right\Vert _{{L^{1 }(\mathbb R^{n}, a^{\frac{\varepsilon}{p}}})}\left( \left\Vert g\right\Vert _{L_a^{p),\theta_1}}+\left\Vert \widehat{g}\right\Vert _{{L_a^{q),\theta_2}}}\right) \notag
\\&=\left\Vert f\right\Vert _{{L^{1 }(\mathbb R^{n}, a^{\frac{\varepsilon}{p}}})}\left\Vert g\right\Vert _{A_{q),\theta_2}^{p),\theta_1}\left( \mathbb R^{n}\right) }.\notag 
\end{align}
It is easy to show that $A_{b,q)}^{a,p)}\left( \mathbb R^{n}\right) $  is a module over $L^{1}(\mathbb R^{n}),$ algebraically.
Hence the space  $A_{q),\theta_2}^{p),\theta_1}\left(\mathbb R^{n}\right)$ is a  Banach convolution module over  ${L^{1 }(\mathbb R^{n}, a^{\frac{\varepsilon}{p}})}$ for ${0<\varepsilon \leq p-1}$.

\end{proof}

\begin{proposition}
Let $a(x)$ be a Beurlin's weight function. Then 

a) The space of multipliers $M\left( L^{1}(\left( \mathbb R^{n}\right),a^{\frac{\varepsilon}p} ),\left(
L_a^{{p),\theta }}\left( \mathbb R^{n}\right) \right) ^{\ast }\right) $
is isometrically isometric to the dual space  $( L_a^{p),\theta}\left(\mathbb R^{n}\right))^{\ast}$ of 
the generalized  grand Lebesgue space $ L_a^{p),\theta}\left(\mathbb R^{n}\right).$ Then

b)  The space of multipliers $M\left( L^{1}(\left( \mathbb R^{n}\right),a^{\frac{\varepsilon}p} ),\left(
A_{q),\theta _{2}}^{p),\theta _{1}}\left( \mathbb R^{n}\right) \right) ^{\ast }\right) $
is isometrically isometric to the dual space $\left( L^{1}\left(  \mathbb R^{n}\right)
\ast A_{q),\theta _{2}}^{p),\theta _{1}}\left(  \mathbb R^{n}\right) \right) ^{\ast }$
of $\ $the space $L^{1}\left( \mathbb R^{n}\right) \ast A_{q),\theta _{2}}^{p),\theta
_{1}}\left( \mathbb R^{n}\right) .$
\end{proposition}

\begin{proof}
a) By Theorem { 8},  the generalized  grand Lebesgue space $ L_a^{p),\theta}\left(\mathbb R^{n}\right)$ is an essentially Banach convolution module  over ${L^{1 }(\mathbb R^{n}, a^{\frac{\varepsilon}{p}})}$ for ${0<\varepsilon \leq p-1}.$ Then by the module factorization theorem (see $[36]$),  $L^{1}\left( \mathbb R^{n}\right) \ast L_a^{p),\theta}\left(\mathbb R^{n}\right)=L_a^{p),\theta}\left(\mathbb R^{n}\right) .$ Thus  From Corollary
2.13 in \cite{ri1}, we obtain
 \begin{align} M\left( L^{1}(\left( \mathbb R^{n}\right),a^{\frac{\varepsilon}p} ),\left(
L_a^{{p),\theta }}\left( \mathbb R^{n}\right) \right) ^{\ast }\right) = ( L_a^{p),\theta}\left(\mathbb R^{n}\right))^{\ast}.
\end{align}

b) Similarly by Theorem $8$, $A_{q),\theta _{2}}^{p),\theta _{1}}\left(
 \mathbb R^{n}\right) $ is a Banach convolution module  over ${L^{1 }(\mathbb R^{n}, a^{\frac{\varepsilon}{p}})}$ for ${0<\varepsilon \leq p-1}.$ Thus ${L^{1 }(\mathbb R^{n}, a^{\frac{\varepsilon}{p}})}
\ast $ $A_{q),\theta _{2}}^{p),\theta _{1}}\left( \mathbb R^{n}\right) \subset $ $
A_{q),\theta _{2}}^{p),\theta _{1}}\left(  \mathbb R^{n}\right) .$ Again from Corollary
2.13 in \cite{ri1}, we have
\begin{equation*}
M\left( L^{1}\left(  \mathbb R^{n}\right) ,\left( A_{q),\theta _{2}}^{p),\theta
_{1}}\left(  \mathbb R^{n}\right) \right) ^{\ast }\right) =\left( L^{1}\left(  \mathbb R^{n}\right)
\ast A_{q),\theta _{2}}^{p),\theta _{1}}\left(  \mathbb R^{n}\right) \right) ^{\ast }.
\end{equation*}
\end{proof}

Remember some definitions and theorems that are necessary for us.
\begin{definition}
Let $(X,\|.\|_X)$ be Banach function space and let $f\in X.$ We say that $f$ has absolutely continuous norm in $X$ if
\begin{align*}
\lim_{n\rightarrow\infty}\|f\chi_{E_n}\|_X=0
\end{align*}
for all $\{E_n\}, n\in\mathbb N$ satisfying $E_n\rightarrow \emptyset.$ We denote by $X_a$ the set of functions in $X$ with absolutely continuous norm. If $X_a=X,$ then X is said to have absolute continuous norm.
\end{definition}
\begin{theorem}
( See \cite{bs} , Corollary $4.3$). The dual space $X^{\ast }$of a
Banach function space $X$ is canonically isometric to the associate space $X^{^{\prime
}} $ if and only if $X$ has absolutely continuous norm.
\end{theorem}

\begin{theorem}
( See \cite{bs} , Corollary $4.4$). A Banach function space is
reflexive if and only if both $X$ and its associate space $X^{^{\prime }}$
have absolute continuous norm.
\end{theorem}

\begin{proposition}
There exists weight function $a(x)$ such that the corresponding space   $L_{a}^{p),\theta}(\mathbb R^n),$   do not have absolutely continuous norm.
\end{proposition}
\begin{proof}
For the proof, it will be enough to find a generalized grand Lebesgue space $L_a^{p),\theta}(\mathbb R^n)$ and a  non-absolute continuous function in this space. Let $f(t)=e^{-\mid t\mid},$ let  $a(t)={(1+\mid t\mid)^{-2}}$ a weight function and let $E_n=(\frac{n}{n+1},1)\subset(0,1)\subset\left( \mathbb R\right) $ for $n\in N.$ Thus $(E_n), n\in\mathbb N$  is a sequence of subsets of $ \mathbb R$. It is clear that $E_n\rightarrow \emptyset.$ 
Since
 $f(t)=e^{-\mid t\mid}\in L_{a}^{p),\theta}(\mathbb R),$ then
 $e^{-\mid t\mid}\chi_{E_n}\in  L_{a}^{p),\theta}(\mathbb R), $ for all $ n\in \mathbb N.$

Also since $(\frac{n}{n+1},1)\subset (0,1),$ 
 we have
\begin{align*}
\|e^{-\mid t\mid}\chi_{E_n}\|_{L_{a}^{p)}}
&>\notag{\varepsilon_0^\theta}(\frac{1}{4})^{\frac{\varepsilon_0}{p(p-\varepsilon_0)}}(\frac{1}{e^{p}})^{\frac{1}{p-\varepsilon_0}}\left[\frac{e^{\varepsilon_0}-e^{\frac{n\varepsilon_0}{n+1}}}{\varepsilon_0}\right] ^{\frac{1}{p-\varepsilon_0 }}\\&
\end{align*}
for any fixed $0<\varepsilon_o\leq p-1.$
It is easy to see that if   $n\rightarrow\infty,$ then
\begin{align*}
{\varepsilon_0^\theta}(\frac{1}{4})^{\frac{\varepsilon_0}{p(p-\varepsilon_0)}}(\frac{1}{e^{p}})^{\frac{1}{p-\varepsilon_0}}\left[\frac{e^{\varepsilon_0}-e^{\frac{n\varepsilon_0}{n+1}}}{\varepsilon_0}\right] ^{\frac{1}{p-\varepsilon_0 }}\rightarrow 0.
\end{align*}
That means
\begin{align*}
\|e^{-\mid t\mid}\chi_{E_n}\|_{L_{a}^{p),\theta}}> 0.
\end{align*}
Hence  $f(t)=e^{-\mid t\mid}$  is not absolutely continuous in    $L_{a}^{p),\theta}(\mathbb R),$  and so  $L_{a}^{p),\theta}(\mathbb R),$   do not have absolutely continuous norm. 
\end{proof}
\begin{corollary}
The generalized grand Lebesgue space  $L_{a}^{p),\theta}(\mathbb R^n)$  is generally not reflexive.
\end{corollary}

\begin{proof}
The proof is clear from Theorem 7 ,Theorem 8 and Proposition 5.
\end{proof}

\begin{corollary}
The dual space $(L_{a}^{p),\theta}(\mathbb R^n))^{\star}$  of the generalized grand  space $L_{a}^{p),\theta}(\mathbb R^n)$  is generally not isometric to the associate space $(L_{a}^{p),\theta}(\mathbb R^n))^{\prime}$ of the generalized grand Lebesgue space.
\end{corollary}

\begin{proof}
The proof of this Corollary is easy from Theorem 7 ,Theorem 8 and Proposition 5.
\end{proof}
\begin{proposition}
There exist weight functions $a(x), b(x)$ such that the corresponding space  $A_{q),\theta_2 }^{p),\theta _{1}}\left( \mathbb  R^{n}\right)$  do not have absolutely continuous norm.
\end{proposition}
\begin{proof}
 Let $b(t)=a(t)={(1+\mid t\mid)^{-2}}$ and let  $E_n=(\frac{n}{n+1},1)\subset(0,1)\subset\left( \mathbb R\right) $ for $n\in N.$ It is clear that $E_n\rightarrow \emptyset$ and   $(1+\mid t\mid^{2})^{-1}>(1+\mid t\mid)^{-2}.$ Take the function $f(t)=e^{-\mid t\mid}.$  We showed in Proposition 4 that  $f(t)=e^{-\mid t\mid}\in L_{a}^{p),\theta_1}(\mathbb R)$ and it is not absolutely continuous in   $ L_{a}^{p),\theta_1}(\mathbb R).$ 
A simple calculation shows that the Fourier transform  $\hat f$ of $f(t)=e^{-\mid t\mid}$  is in $ L_b^{2),\theta_2}(\mathbb R).$  Thus $f\in A_{2),\theta_2 }^{p),\theta _{1}}\left(  \mathbb R^{n}\right).$
 From the definition of the norm of the space   $A_{2),\theta_2 }^{p),\theta _{1}}\left(\mathbb  R^{n}\right)$  we obtain
\begin{align}
\left\Vert e^{-\mid t\mid}\chi_{E_n}\right\Vert _{A_{2),\theta_2}^{p),\theta_1}} =\left\Vert
e^{-\mid t\mid}\chi_{E_n}\right\Vert _{L_a^{p),\theta_1}}+\left\Vert \widehat{e^{-\mid t\mid}\chi_{E_n}}\right\Vert _{L_b^{2),\theta_2}}>\left\Vert e^{-\mid t\mid}\chi_{E_n}\right\Vert _{L_a^{p),\theta_1}}>0. 
\end{align}
Hence  $f(t)$ is not absolutely continuous in    $A_{2),\theta_2 }^{p),\theta _{1}}\left( \mathbb R^{n}\right)$  and so  $A_{2),\theta_2 }^{p),\theta _{1}}\left( \mathbb R^{n}\right)$    do not have absolutely continuous norm. 

\end{proof}
\begin{corollary}
The space   $A_{q),\theta_2 }^{p),\theta _{1}}\left( \mathbb R^{n}\right)$   is generally not reflexive.
\end{corollary}

\begin{proof}
The proof is clear from Theorem 7 ,Theorem 8 and Proposition 6.
\end{proof}

\begin{corollary}
The dual space $( A_{q),\theta_2 }^{p),\theta _{1}}\left(\mathbb  R^{n}\right)  )^{\star}$  of the space $ A_{q),\theta_2 }^{p),\theta _{1}}\left( \mathbb R^{n}\right) $  is generally not isometric to
the associate space $( A_{q),\theta_2 }^{p),\theta _{1}}\left(\mathbb  R^{n}\right))^{\prime}  $  of the space $ A_{q),\theta_2 }^{p),\theta _{1}}\left( \mathbb R^{n}\right) $.
\end{corollary}

\begin{proof}
The proof of this theorem is easy from Theorem 7 ,Theorem 8 and Proposition 6.
\end{proof}

\end{document}